\documentclass[reqno,keywordsasfootnote,10pt]{article}
\usepackage[english]{babel}
\usepackage[latin1]{inputenc}
\usepackage{amsmath,amsthm,amsfonts,amssymb,times,epsfig}
\numberwithin{equation}{section}

  \usepackage{epstopdf}
\newtheorem{thm}{Theorem}[section]
\newtheorem{lem}[thm]{Lemma}
\newtheorem{prop}[thm]{Proposition}

\newtheorem{df}[thm]{Definition}

\newtheorem{ass}[thm]{Assumptions}

\newcommand{\E}{\mathbb{E}}

\newcommand{\prob}{\mathbb{P}}

\newcommand{\Ppar}{\mathcal{P}_r}

\newcommand{\R}{\mathbb{R}}
\newcommand{\Z}{\mathbb{Z}}
\newcommand{\N}{\mathbb{N}}

\newcommand{\F}{\mathcal{F}}

\newcommand{\distrib}{\overset{\mbox{$\mathcal D$}}{=}}

\newcommand{\hilbert}{\mathcal{H}}

\def\be{\begin{equation} }
\def\ee{\end{equation} }
\begin{document}
\date{\today}
\title{Competing Particle Systems and the Ghirlanda-Guerra Identities}
\author{\Large Louis-Pierre Arguin \\[2ex] 
Weierstrass Institute for Applied Analysis and Stochastics
\\Berlin, Germany\thanks{arguin@wias-berlin.de}
}%
\maketitle  

\begin{abstract}
We study point processes on the real line whose configurations $X$ can be ordered decreasingly and evolve by increments which are functions of correlated gaussian variables.   The correlations are intrinsic to the points and quantified by a matrix $Q=\{q_{ij}\}$. 
Quasi-stationary systems are those for which the law of $(X,Q)$ is invariant under the evolution up to translation of $X$. 
It was conjectured by Aizenman and co-authors that the matrix $Q$ of robustly quasi-stationary systems must exhibit a hierarchal structure. 
This was established recently, up to a natural decomposition of the system, whenever the set $S_Q$ of values assumed by $q_{ij}$ is finite. In this paper, we study the general case where $S_Q$ may be infinite. Using the past increments of the evolution, we show that the law of robustly quasi-stationary systems must obey the Ghirlanda-Guerra identities, which first appear in the study of spin glass models. This provides strong evidence that the above conjecture also holds in the general case. 

 \\[1ex]
(2000 Mathematics Subject Classification: 60G55, 60G10, 82B44)\\
\vspace*{1cm}
\end{abstract}
\newpage 
\tableofcontents

\section{Introduction}
\subsection{Background}
Competing particle systems are point processes $X=\{X_i\}$ on $\R$ whose configurations can be ordered in decreasing order $X_1\geq X_2\geq ...$. 
We study a dynamics of $X$ where the particles compete in the sense that, at each time step, the positions are evolved by increments whose correlations depend on intrinsic characteristics of the points. Precisely, we assign to each $X$ a covariance or {\it overlap} matrix $Q=\{q_{ij}\}$. The overlap $q_{ij}$ quantifies the similarity between the $i$-th point and the $j$-th point. We set the overlap to $1$ when the particles are identical i.e. $q_{ii}=1$ for all $i$. As $Q$ is a covariance matrix, it follows that $|q_{ij}|\leq 1$. The overlaps are not affected by the dynamics and are simply permuted under evolution. Precisely, let $\psi$ be some real function, the dynamics is $(X, Q) \mapsto (\widetilde X, \widetilde Q)$
with 
\begin{eqnarray}
\widetilde X_i &=& X_{\pi(i)} +\psi(\kappa_{\pi(i)}) \nonumber 
  \\      
\widetilde q_{ij} &=& q_{\pi(i)\pi(j)} \, ,
\label{eq:ev_XQ}
\end{eqnarray}  
where $\pi$ is a permutation of $\N$ which reorders $\widetilde X_i$ and $\kappa$ is a gaussian field independent of $X$ with covariance given by an entry-wise power of $Q$. 

The question of interest is to characterize the distributions on the pair $(X,Q)$ which are quasi-stationary in the sense that the joint law of the gaps of $X$ and $Q$ is invariant under the stochastic evolution \eqref{eq:ev_XQ} (see also \cite{Chat_Pal} and \cite{Pal_Pitman} for related setups). The uncorrelated case where $Q$ is the identity was handled in \cite{RA}. Under mild assumptions on $X$, it was shown that quasi-stationarity implies that the statistics of the gaps are those of a Poisson process on $\R$ with exponential density. The correlated case was first studied in \cite{argaiz}. It was proven that, under some robustness conditions on the quasi-stationary property and up to a natural decomposition of the system, $Q$ must exhibit a hierarchal structure whenever the state space of the overlaps was finite i.e. the possible values taken by $q_{ij}$. The aim of this paper is to provide evidence that the hierarchal structure is also necessary for quasi-stationarity to hold when the state space is infinite. Namely, we establish that $Q$ must satisfy constraining identities which are consistent with the hierarchal structure. These identities are known as the Ghirlanda-Guerra identities in statistical mechanics \cite{GG}.

For our purpose, we can assume that $X$ has infinitely many particles a.s. because no finite systems can be quasi-stationary due to the spreading of the gaps under evolution \cite{RA}.
As in \cite{argaiz}, we restrict ourselves to $X$ for which there exists $\beta>0$ such that $\sum_i e^{\beta X_i}<\infty \text{ a.s.}$
In this case, one can see $(X,Q)$ as a Random Overlap Structure or ROSt $(\xi,Q)$ by mapping $X$ to the exponentials of the position:
\be
\xi_i=\frac{e^{\beta X_i}}{\sum_i e^{\beta X_i}}\ .
\label{eqn rost}
\ee
\begin{df}
A ROSt is a random variable on the space $\Omega_{os}:=P_m\times \mathcal{Q}$ where $P_m$  is the space of sequences $(s_i,i\in\N)$ such that $s_1\geq s_2\geq ...\geq 0$ with $\sum_i s_i\leq 1$ and $\mathcal{Q}$ is the space of positive semi-definite symmetric matrices with $1$ on the diagonal.
\end{df}
The space $\Omega_{os}$ is equipped with the uniform topology on the sequences $s$ together with the topology on $\mathcal{Q}$ inherited from the product topology on $[-1,1]^{\N\times\N}$. This renders the space $\Omega_{os}$ compact and separable (see \cite{argaiz} for details).
From the ROSt perspective, we may assume that $Q$ is supported on positive definite matrix i.e. that $|q_{ij}|<1$. Indeed, we simply identify two particles $i$ and $j$ for which $q_{ij}=1$ and add their weight. From \eqref{eqn rost}, we see that the competitive evolution \eqref{eq:ev_XQ} becomes
\begin{equation}
(\xi,Q)\mapsto \Phi_{\psi(\kappa)}(\xi,Q):=\left(\left(\frac{\xi_ie^{\psi(\kappa_i)}}{\sum_j \xi_je^{\psi(\kappa_j)}},i\in\N\right)_\downarrow,\pi\circ Q\circ\pi^{-1}\right)\ .
\label{evolution rmpr1}
\end{equation}
Again, $\pi$ is the reshuffling induced by the mapping and the symbol $\downarrow$ means that the weights are reordered in decreasing order after evolution. The evolved weights are normalized to sum up to $1$. For simplicity, we will sometimes drop the dependence on $\psi$ and write $\Phi_r$ for the mapping \eqref{evolution rmpr1} where $\kappa$ has covariance $Q^{*r}$, the $r$-th entry-wise power of $Q$. Since the normalized weights depend only on the gaps of $X$, quasi-stationarity of $(X,Q)$ under \eqref{eq:ev_XQ} translates into the invariance of the law of $(\xi,Q)$ under $\Phi_r$.
\begin{df}
Fix $\psi:\R\to\R$.
A ROSt $(\xi,Q)$ is quasi-stationary under $\Phi_r$ if 
$$ \Phi_r(\xi,Q)\distrib (\xi,Q)$$
where the symbol $\distrib$ means equality in distribution.
It is said to be robustly quasi-stationary if it is quasi-stationary under $\Phi_{r}$ for infinitely many $r\in\N$.
\end{df} 
Note that quasi-stationary ROSt's must satisfy $\sum_i\xi_i=1$ a.s. due to renormalization of the weights. As $\Omega_{os}$ is compact and separable, one can decompose a quasi-stationary ROSt under $\Phi_r$ into ergodic ROSt's for which the only functions $f:\Omega_{os}\to\Omega_{os}$ satisfying
$ \E_r\left[f\left(\Phi_r(\xi,Q)\right)\big| \xi,Q\right]=f(\xi,Q)$ a.s. are the constants.

A sufficient condition for the evolution \eqref{evolution rmpr1} to be non-singular and for $\Phi_r(\xi,Q)$ to be a ROSt is the finiteness of the expectation of $e^{\psi(\kappa)}$.
Throughout this paper, $\psi$ will be fixed and assumed to belong to the following class of functions which ensures this condition. This class also allows a good control on the evolution.
 \begin{ass}
The function $\psi:\R\to\R$ is in $C^2(\R)$ with bounded derivatives.
Furthermore, for $Y$ a standard gaussian variable, the law of $\psi(Y)$ is absolutely continuous with respect to the Lebesgue measure.
\label{intro ass}
\end{ass}


\subsection{Main Results}
The only known examples of quasi-stationary ROSt's for all $\psi$ satisfying Assumption \ref{intro ass} are given by the so-called Ruelle Probability Cascades or RPC's \cite{Ruelle,BS, AS22}. The RPC's are constructed from Poisson-Dirichlet variables and the Bolthausen-Sznitman coalescent.
This coalescent is a Markov process $\Gamma=(\sim_t, \ t\geq 0)$ on the space of equivalence relations on $\N$ for which $i\sim_t j$ implies $i\sim_s j$ for all $s\geq t$. For more on these processes, the reader is referred to \cite{bertoin, BS}.
\begin{df}
Let $x:q\mapsto x(q)$ be a distribution function on $[0,1]$ with $x(1^-)\neq 1$. A RPC with parameter $x$ is the ROSt $(\xi,Q)$ where $\xi$ is a Poisson-Dirichlet variable $PD(x(1^-),0)$ and $Q$ is as follows. Let $\Gamma$ be a Bolthausen-Sznitman coalescent independent of $\xi$. Then
$$ q_{ij}= x^{-1}(e^{-\tau_{ij}})$$
where $\tau_{ij}:=\min\{t \ :  i\sim_t j \}$ and $x^{-1}$ is the right-continuous inverse of $x$. In particular, $\prob(q_{ij}\leq q)=x(q)$ for all $i\neq j$.
\label{df rpc}
\end{df}
It was conjectured by Aizenman {\it et al} that the RPC's were the only ROSt's that are quasi-stationary in a "robust" sense, where the notion of robustness was to still be determined \cite{AS22}.
The striking point of the conjecture, if proven true, is the necessity of hierarchal correlations for stability under competitive evolution. Indeed, the RPC inherits a hierarchal structure from the coalescent i.e.  
\begin{equation}
 (q_{ij}=q \text{ and } q_{jk}=r) \Longrightarrow q_{ik}=\min\{q,r\}\ .
 \label{eqn ultra}
 \end{equation}
 
A proof of a version of the conjecture was given in \cite{argaiz} for systems with finite state-space i.e. for which the random set $S_Q := \{q_{ij}: 1\leq i<j <\infty\} $ is finite a.s. Such systems can be decomposed into subsystems called {\it $Q$-factors} for which the sets $S_Q(i):=\{q_{ij}: j\neq i\}$ are identical for each $i$. It was proven that if $(\xi,Q)$ is robustly quasi-stationary and ergodic for all multiples of a smooth function $\psi$, then each of its $Q$-factors is a RPC. Our first result is to show that the decomposition into $Q$-factors is not necessary whenever quasi-stationarity is assumed under $\psi(\beta\kappa+h)$ for all $\beta\geq 0$ amd for all $r\in\N$ provided $\psi'(h)\neq 0$.
\begin{thm}
Let $h\in\R$ be such that $\psi'(h)\neq 0$. If a ROSt with finite state space is quasi-stationary and ergodic under $\Phi_r$ with function $\psi(\beta \kappa+h)$ for all $\beta\geq 0$ and for all $r\in\N$, then it is a RPC. In particular, $Q$ satisfies \eqref{eqn ultra} almost surely.
\label{thm main1}
\end{thm}

In the general case where $S_Q$ may be infinite, it was shown in \cite{argaiz} that:
\begin{thm}[Theorem 4.2 in \cite{argaiz}]
Let $(\xi,Q)$ be a ROSt that is robustly quasi-stationary and ergodic for some function $\psi$ satisfying Assumption \ref{intro ass}. The following hold:
\begin{enumerate}
\item $\xi$ is a Poisson-Dirichlet variable independent of $Q$;
\item $Q$ is directed by a random probability measure $\mu$ on a Hilbert space $\hilbert$:  \\for $i\neq j$, $
q_{ij}=(\phi_i,\phi_j)$ where $(\phi_i,i\in\N)$ are iid $\mu$-distributed.
\end{enumerate}
\label{thm general}
\end{thm}
In the case of finite state space, the directing measure is discrete. It is then possible to carry an induction argument on the cardinality of the state space to prove that the directing measure is again a cascade.
In the present paper, we provide strong identities that must be generally satisfied by the directing measure of a quasi-stationary ROSt. Our main result is:
 \begin{thm}
Let $h\in\R$ be such that $\psi'(h)\neq 0$. Consider a ROSt that is quasi-stationary and ergodic under $\Phi_r$ with function $\psi(\beta\kappa+h)$ for all $\beta$ in an interval containing $0$ and for every $r\in\N$. Then, its directing measure $\mu$ satisfies
\begin{equation}
\E\left[ \bigotimes_{t=1}^s\mu\left(q_{s,s+1}\in A \right)\ \Big| \F_s\right] =\frac{1}{s}\E\left[\mu\otimes\mu\left(q_{12}\in A\right)\right]+\frac{1}{s}\sum_{l=1}^{s-1}\chi_A(q_{ls})
\label{gg1}
\end{equation}
for every $s\in\N$
where $A\subseteq [-1,1]$, $\chi_A$ is the identity function of the set $A$ and $\F_s$ is the $\sigma$-field generated by the Gram matrix of $s$ vectors.
\label{thm main2}
\end{thm}
More generally, we obtain an identity for the $r$-th moment whenever $(\xi,Q)$ is invariant under $\Phi_r$. In the case where quasi-stationarity holds for every $\Phi_r$, these determine the conditional distribution.
The identities \eqref{gg1} are known as the Ghirlanda-Guerra identities in the study of spin glass models \cite{GG, CG}. It is a non-trivial fact that they arise in the general setting of competing particle systems. 
They are satisfied by the RPC's and hence consistent with hierarchal overlaps.
In fact, the Ghirlanda-Guerra identities have a simple interpretation:
conditionally on the inner product of $s$ vectors $q_{12}$ ... $q_{s-1,s}$, the inner product of an additional vector drawn under $\mu$ with a previous one is independent of the given frame with probability $1/s$ or takes the value $q_{ls}$, $1\leq l\leq s-1$, each with probability $1/s$. 

The main concept used to derive Theorems \ref{thm main2} and \ref{thm main1} is the so-called {\it past velocity}. Precisely, in Section 2, we consider independent time-steps of the evolution $\Phi_r$ keeping track of the past time-steps. The past velocity is simply defined as the time-average of the past increments. It is shown to exist and to be common to all particles whenever the system is quasi-stationary. Similarly as in \cite{argaiz}, the study of the evolution for a generic $\psi$ can be reduced to a linear $\psi$ by a Central Limit Theorem argument as explained in Appendix \ref{app var}. It turns out that the collection of velocities obtained from the different linear evolutions single out the parameter of the RPC thereby yielding Theorem \ref{thm main1}. In Section 3, we used the fact that the velocity is common (and deterministic for ergodic systems) to conclude that the distribution of $Q$ satisfies the Ghirlanda-Guerra identities. The argument is very similar to the proof of these identities for spin glass models in the sense that the common velocity plays the role of the self-averaging of the free energy. Along the way, we also prove that quasi-stationary ROSt's obey the so-called Aizenman-Contucci identities, which can be seen as a weaker version of the Ghirlanda-Guerra identities \cite{AC}.


\section{The Past Velocity}
\subsection{Definition}
The past velocity naturally appears when re-expressing the evolution \eqref{evolution rmpr1} as a deterministic mapping on a space that includes the past and future increments of the evolution. 

Let $\nu_{Q^{*r}}$ be the law of the gaussian field $\kappa$ with covariance $Q^{*r}$ and $\prob$ the law of some ROSt. We consider $\prob_r$ the probability measure on $\Omega_{os}\times \prod_{t=0}^\infty \R^\N$ consisting of $\prob$, coupled through $Q$, with independent copies of the field:
\begin{equation}
d\prob_r=d\prob(\xi,Q)\times\prod_{t\geq 0} d\nu_{Q^{*r}}(\kappa(t)).
\label{forw}
\end{equation}
Clearly, the future increments $(\kappa(t),t\geq0)$ are exchangeable given $(\xi,Q)$ as they are simply iid. 
We are interested in extending the probability measure $\prob_r$ in a consistent way to include the past increments $(\kappa(t),t<0)$ and thus get a probability measure on 
$$\Omega:=\Omega_{os}\times\prod_{t\in\Z} \R^\N.$$ 
The relevant dynamics on the space $\Omega$ is the evolution \eqref{evolution rmpr1} on $(\xi,Q)$ together with a time-shift of the fields. We stress that the field $\kappa$ must also be reindexed after evolution.
\begin{df}
Let $\Phi_{\psi(\cdot)}$ be of the form \eqref{evolution rmpr1}. We define the mapping $\Lambda: \Omega\to\Omega$
$$ \Lambda(\omega)=\Lambda(\xi,Q,(\kappa(t),t\in\Z)):=\left(\Phi_{\psi(\kappa(0))}(\xi,Q), (\kappa_\downarrow(t+1),t\in\Z)\right)$$
where $\downarrow$ stands for the reindexing of the gaussian field with respect to the ordering of the points after evolution by $\Phi_{\psi(\kappa(0))}$.
\end{df}

It is shown in Appendix \ref{app ext} that the extension of $\prob_r$ to $\Omega$ exists whenever the system is quasi-stationary. Furthermore, similarly as for the future increments, the sequence of past increments is exchangeable conditionally on $(\xi,Q)$.
 \begin{lem}[Appendix \ref{app ext}]
Let $(\xi,Q)$ be a quasi-stationary ROSt under $\Phi_r$ for some $r\in\N$. There exists a unique $\Lambda$-invariant probability measure on $\Omega$ whose restriction on $\Omega_{os}$ is the law of $(\xi,Q)$. This measure is ergodic under $\Lambda$ if and only if $(\xi,Q)$ is ergodic.

Moreover, the sequence of past increments $(\kappa(t),t<0)$ is exchangeable under this probability measure conditionally on $(\xi,Q)$.
\label{lem exist}
\end{lem}
From now on, we will also write $\prob_r$ for the extension of the probability measure \eqref{forw} to $\Omega$. 
\begin{df}
The past velocity of the $i$-th point is the time-average of its past increments i.e. for $\omega\in\Omega$
\be
v_i(\omega):=\lim_{T\to\infty}\ \frac{1}{T}\sum_{t=1}^T\psi(\kappa_i(-t)) .
\label{df velocity}
\ee
\end{df}
It is important to bear in mind that the velocity is in essence very different from the time-average of the future increments due to the reordering. Indeed, the $i$-th point moved in front of all but $i-1$ points during the course of the competitive evolution. Thus its past increments are by nature atypical.
The existence of the limit \eqref{df velocity} is a simple consequence of the exchangeability of the increments.
\begin{prop}
Let $(\xi,Q)$ be a quasi-stationary ROSt under $\Phi_r$ for some $r\in\N$. For all $i\in\N$, the limit $v_i(\omega)$ exists $\prob_r$-a.s. and $v_i(\omega)\in L^p(\prob_r)$ for any $1\leq p<\infty$. 

Moreover, the velocity is an intrinsic quantity of a particle in the sense that
\begin{equation}
v_i(\omega)=v_{\pi(i)}(\Lambda\omega)
\label{evol veloc}
\end{equation}
where $\pi$ is the permutation induced by the evolution $\omega\mapsto \Lambda \omega$.
\label{veloc exist}
\end{prop}
\begin{proof}
By de Finetti's theorem and the exchangeability of the past increments asserted in Lemma \ref{lem exist}, the fields $(\kappa(t),t<0)$ are iid given $(\xi,Q)$ and $\alpha$, the empirical distribution of $(\kappa(t), t<0)$. On the other hand, it is proven in  Lemma \ref{lem2 exist} of Appendix \ref{app ext} that $\E_{r}\big[ \ |\psi(\kappa_i(-1))| \ \big|\xi,Q,\alpha\big]<\infty$ a.s. Thus the first claim follows by the law of large numbers. Second, by a combination of Jensen's inequality, Fatou's lemma and exchangeability, we have
$$\E_r[ \ |v_i(\omega)|^p \ ]\leq \E_{r}\left[|\psi(\kappa_i(-1))|^p\right].$$
which is also finite by the proof of Lemma \ref{lem2 exist}. 
The equality \eqref{evol veloc} is clear as the past velocity depends only on increments in the distant past.
\end{proof}

\subsection{The velocity is common}
We now make rigorous the intuitive idea that the points must share a common velocity for the system to be stable.
\begin{prop}
If $(\xi,Q)$ is a quasi-stationary ROSt under $\Phi_r$ for all functions $\lambda\psi$, $\lambda$ in some open set of $\R$. Then $v_{i}(\omega)\equiv v(\omega)$
for all $i\in\N$ $\prob_r$-a.s. 

If it is ergodic, then the past velocity is deterministic and
$$ v(\omega)=\E_{r}\left[\sum_i\xi_i\text{ } \psi(\kappa_i(-1))\right]\text{  $\prob_{r}$-a.s.}.$$
\label{prop common}
\end{prop}

Before proving the proposition, we need to introduce the generating function of the cumulants of the past increments. 
Let $(\xi,Q)$ be a ROSt. For $\lambda\in\R$, we set
\be
 \Ppar(\lambda):=\E_r\left[\log \sum_i \xi_i e^{\lambda\psi(\kappa_i(0))}\right].
 \label{df P}
\ee
$\Ppar(\lambda)$ is well-defined in the case $\psi$ satisfies Assumption \ref{intro ass} since by Jensen's inequality
$$0 \leq \Ppar(\lambda)\leq \log \int_\R \frac{e^{-z^2/2}}{\sqrt{2\pi}}\text{ }e^{\lambda \psi(z)}dz\ .$$
In the case where $(\xi,Q)$ is quasi-stationary, we have for all $T\in\N$
\begin{equation}
\Ppar(\lambda)=\frac{1}{T}\E_r\left[\log \sum_i \xi_i e^{\lambda\sum_{t=0}^{T-1} \psi(\kappa_i(t))}\right].
\label{stat P}
\end{equation}

The function $\Ppar(\lambda)$ is a good tool to compare the past increments of a point $i$ with the $\xi$-averaged increment of the crowd.
\begin{lem}
Let $(\xi,Q)$ be a quasi-stationary ROSt under $\Phi_r$ for all functions $\lambda\psi$, $\lambda$ in some open set of $\R$. Define $S_i(T;\omega):=\frac{1}{T}\sum_{t=1}^{T}\psi(\kappa_i(-t))$ and $\langle S(T)\rangle_\omega:=\sum_i\xi_i\text{ }S_i(T;\omega)$. 
Then for all $T\in\N$
$$\frac{d}{d\lambda}\Ppar(\lambda)=\E_{r}\left[\langle S(T)\rangle_\omega\right]=\E_{r}\left[\sum_i \xi_i \text{ }\psi(\kappa_i(-1))\right]$$
and 
$$\frac{1}{T}\frac{d^2}{d\lambda^2}\Ppar(\lambda)=\E_{r}\left[\sum_i \xi_i \left(S_i(T;\omega)-\langle S(T)\rangle_\omega\right)^2\right].$$
In particular,
\begin{equation}
\lim_{T\to\infty}\E_{r}\left[\sum_i \xi_i \left(S_i(T;\omega)-\langle S(T)\rangle_\omega\right)^2\right]=0 \ .
\label{eqn moments}
\end{equation}
\label{moments}
\end{lem}
\begin{proof}
The two expression of the derivatives are obtained by simply taking derivatives in $\lambda$ on both sides of equation \eqref{stat P}. The condition that quasi-stationarity holds for $\lambda$ in an open set of $\R$ is necessary for the identity \eqref{stat P} to hold in a neighborhood of the point where the derivative is taken. Linearity of expectation and the exchangeability of the increments yield the second equality for $\frac{d}{d\lambda}\Ppar(\lambda)$. 
The limit $T\to\infty$ follows directly from the fact that $\Ppar(\lambda)$ has a finite second derivative.
\end{proof}

\begin{proof}[Proof of Proposition \ref{prop common}]
We claim that there exists a sequence $T_n\in\N$ such that for all $i\in\N$ as $n\to\infty$ 
\begin{equation}
\big|S_i(T_n;\omega)-\langle S(T_n)\rangle_\omega\big|\to  0 \text{  $\prob_{r}$-a.s.}
\label{common eqn}
\end{equation}
Indeed, it follows from equation \eqref{eqn moments} of Lemma \ref{moments} that as $T\to\infty$
$$\sum_i \xi_i \left(S_i(T;\omega)-\langle S(T)\rangle_\omega\right)^2\to  0 \text{  in $L^1(\prob_{r})$}\ .$$
This ensures the existence of the subsequence for which the convergence \eqref{common eqn} holds $\prob_r$-a.s. for all $i\in\N$.
On this subsequence, we also have that $S_i(T_n;\omega)\to v_i(\omega)$ a.s. for every $i\in\N$ by Proposition \ref{veloc exist}.
We conclude that $v_i(\omega)=\lim_{n\to\infty}\langle S(T_n)\rangle_\omega$ a.s. The first part of the proposition is proven.
 
As the past velocity is common, the following equality holds by equation \eqref{evol veloc}
$$ v(\omega)=\frac{1}{T}\sum_{t=0}^{T-1}v(\Lambda^t \omega)\ .$$
We now take the limit $T\to\infty$. Birkhoff's ergodic theorem can be applied as $v(\omega)\in L^1(\prob_{r})$ and we conclude that $v(\omega)=\E_{r}[v(\omega)]$ $\prob_r$-a.s. whenever $(\xi,Q)$ is ergodic.
Furthermore, by dominated convergence,
$$ \E_{r}[v(\omega)]=\E_{r}\left[\lim_{n\to\infty}\langle S{(T_n)}\rangle_\omega\right]=\E_{r}\left[\sum_i\xi_i\text{ }\psi(\kappa_i(-1))\right]$$
where we have used the expression of the first moment in Lemma \ref{moments}.
\end{proof}

\subsection{Velocity and Decomposability}
The velocity and the generating function $\Ppar(\lambda)$ take a simple form when the evolution $\Phi_r$ is governed by a linear function.
\begin{lem}
Let $(\xi,Q)$ be a quasi-stationary ROSt under the evolution $\Phi_r$ for all linear functions $\psi(\kappa)=\lambda\kappa$, $\lambda$ in some open set of $\R$. One has
\begin{equation}
\mathcal{P}_r(\lambda)=\frac{\lambda^2}{2}\int_{-1}^1(1-q^r)\text{ }dx(q)
\end{equation}
where $x(q)$ is the $\xi$-sampled distribution function $\E\left[\sum_{i,j}\xi_{i}\xi_{j} \ \chi_{\{q_{ij}\leq q\}}\right]$. 

In particular, if $(\xi,Q)$ is ergodic then
\begin{equation}
v(\omega)=\lambda\int_{-1}^1(1-q^r)\text{ }dx(q) \ \text{ $\prob_{r}$-a.s.}
\label{eqn velocity}
\end{equation}
\label{ppar}
\end{lem} 
\begin{proof}
We take the derivatives of \eqref{df P} using the gaussian differentiation formula (see e.g. Appendix A in \cite{AS22})
$$ 
\frac{d}{d\lambda}\mathcal{P}_r(\lambda)=\lambda\left(1-\E_{r}\left[\frac{\sum_{i,j}\xi_i\xi_j e^{\lambda\kappa_i(0)}e^{\lambda\kappa_j(0)}\text{ }q^r_{ij}}{\sum_{i,j}\xi_i\xi_j e^{\lambda\kappa_i(0)}e^{\lambda\kappa_j(0)}}\right]\right)
.$$
As the ROSt is quasi-stationary, the right-hand side simply becomes 
$$\lambda\left(1-\E\left[\sum_{i,j}\xi_i\xi_j \text{ }q^r_{ij}\right]\right)=\lambda\int_{-1}^1(1-q^r) \text{ }dx(q)\ .$$
Integration over $\lambda$ yields the first assertion.
The second is obtained from Lemma \ref{moments} and Proposition \ref{prop common}.
\end{proof}

We remark that the full collection of velocities of the evolutions $\Phi_r$, $r\in\N$, singles out the probability measure $dx(q)$ because it determines all the moments. This simple observation is applied to prove Theorem \ref{thm main1}.
\begin{proof}[Proof of Theorem \ref{thm main1}]
By Lemma \ref{lem reduction} proven in Appendix \ref{app var}, $(\xi,Q)$ must be quasi-stationary under the evolutions $\Phi_r$ for all linear functions and for all $r\in\N$. In particular, we can conclude from Theorem 4.4 in \cite{argaiz} that  the $Q$-factors of $(\xi,Q)$ must be RPC's. On the other hand, the velocities of each point must be common and deterministic by Proposition \ref{prop common}. In particular, the velocities of each $Q$-factor must correspond for every evolution $\Phi_r$. We deduce that the measure $dx(q)$ of each $Q$ factor is the same since the collection of velocities determines the moments by equation \eqref{eqn velocity}. Recall from Definition \ref{df rpc} that the parameter $x(q)$ characterizes the law of a RPC. We conclude that $(\xi,Q)$ has only one $Q$-factor and the claim follows.
\end{proof}


\section{The Distributional Identities}
We present the proof of Theorem \ref{thm main2} in this section. In essence, the Ghirlanda-Guerra identities follow from the fact that the velocity is common to all particles and deterministic when the system is ergodic under the considered evolutions. This property can be seen as the equivalent of the self-averaging of the free energy for spin glass models. 
As a first step, we remark that quasi-stationary systems satisfy the weaker Aizenman-Contucci identities which were derived prior to the Ghirlanda-Guerra identities for spin glasses \cite{AC}.

\subsection{The Aizenman-Contucci identities}
It is convenient to introduce a notation for the $\xi$-sampled measure on overlaps. Namely, let $F_s(q)$ be a bounded measurable function on the overlaps of $s$ points, we write
$$ \E^{(s)}[F_s(q)]:=\E\left[\sum_{i_1,...,i_s}\xi_{i_1}...\xi_{i_s}\ F_s(\{q_{i_l,i_{l'}}\}_{l<l'})\right]\ . $$
Plainly, such expectation is invariant under evolution for quasi-stationary ROSt's e.g. for linear $\psi$
\begin{equation}
\E\left[\frac{\sum_{i_1,...,i_s}\xi_{i_1}e^{\lambda\kappa_{i_1}}...\xi_{i_s}e^{\lambda\kappa_{i_s}}\text{ } F_s(q)}{\sum_{i_1,...,i_s}\xi_{i_1}e^{\lambda\kappa_{i_1}}...\xi_{i_s}e^{\lambda\kappa_{i_s}}}\right]=\E^{(s)}[F_s(q)].
\label{AC eqn}
\end{equation}
In particular, the right-hand side of the above equation does not depend on $\lambda$. This simple fact yields moment relations for quasi-stationary ROSt's.
\begin{prop}
Let $h\in\R$ be such that $\psi'(h)\neq 0$.
If $(\xi,Q)$ is a quasi-stationary ROSt under $\Phi_r$ with function $\psi(\beta\kappa+h)$ for all $\beta$ in an interval containing $0$, then for any $s\in\N$, its law satisfies
$$
\frac{s-1}{2}\ \E^{(2)}\left[q^{r}_{12}F_s(q)\right]=s\ \E^{(s+1)}\left[q^{r}_{s,s+1}F_s(q)\right]-\frac{s+1}{2}\ \E^{(s+2)}\left[q_{s+1,s+2}^{r}F_s(q)\right]
\label{eqn prop AC}
$$

\label{prop AC}
\end{prop}
\begin{proof}
By Lemma \ref{lem reduction}, $(\xi,Q)$ is quasi-stationary under $\Phi_{r}$ for all linear functions $\lambda\kappa$ in an interval containing $0$. Therefore, equation \eqref{AC eqn} holds for these $\lambda$. Straightforward gaussian differentiation with respect to $\lambda$ on both sides of \eqref{AC eqn} yields the desired relation.
\end{proof}
The above is a slight generalization of the Aizenman-Contucci identities derived for mean-field spin glass models where $F_s$ is a polynomial \cite{AC}. It is a simple exercise to check that these identities are implied by the Ghirlanda-Guerra identities (see e.g. \cite{GG}).
Therefore, one could ask what extra condition should the system fulfill in order to satisfy the latter. It turns out that ergodicity suffices.

\subsection{The Ghirlanda-Guerra identities}
The key lemma used in the proof of our main result is a factorization of the expectation for observables of a specific form.
A similar factorization was used in the case of spin glass systems to prove the Ghirlanda-Guerra identities (see equation (12) in \cite{GG}).
\begin{lem}
Let $(\xi,Q)$ be a ROSt that is quasi-stationary and ergodic under $\Phi_r$ for all linear function $\psi(\kappa)=\lambda\kappa$ for $\lambda$ in an interval containing $0$. Consider $F_s(q)$ a bounded  function on the overlaps of $s$ points. Then the following holds
\begin{equation}
\E_r\left[\sum_{i_1,...,i_s}\xi_{i_1}...\xi_{i_s}\text{ } \kappa_{i_1}(-1)F_s(q)\right]
=\E_r\left[\sum_{i}\xi_{i}\text{ } \kappa_{i}(-1)\right]\E^{(s)}\left[F_s(q)\right].
\label{eqn factor}
\end{equation}
\label{factorization}
\end{lem}

\begin{proof}
The exchangeability in time of the past increments yields
\begin{equation}
\E_{r}\left[\sum_{i_1,...,i_s}\xi_{i_1}...\xi_{i_s}\text{ } \kappa_{i_1}(-1)F_s(q)\right]=\E_{r}\left[\sum_{i_1,...,i_s}\xi_{i_1}...\xi_{i_s}\left(\frac{1}{T}\sum_{t=1}^T\kappa_{i_1}(-t)\right)F_s(q)\right]
\label{eqn factorization}
\end{equation}
for all $T\in\N$. Recall that $F_s$ is bounded, say $|F_s(q)|\leq C$ for some $C>0$, so
$$\Big|\E_{r}\left[\sum_{i_1,...,i_s}\xi_{i_1}...\xi_{i_s}\left(\frac{1}{T}\sum_{t=1}^T\kappa_{i_1}(-t)\right)F_s(q)\right]\Big|\leq C \E_{r}\left[\sum_i\xi_i\text{ }|\kappa_{i}(-1)|\right] $$
which is finite by Lemma \ref{lem2 exist}. Therefore we can take the limit $T\to\infty$ of equation \eqref{eqn factorization} and by dominated convergence we get
\begin{align*}
\lim_{T\to\infty}\E_{r}\left[\sum_{i_1,...,i_s}\xi_{i_1}...\xi_{i_s}\left(\frac{1}{T}\sum_{t=1}^T\kappa_{i_1}(-t)\right)F_s(q)\right]=\E_{r}\left[\sum_{i}\xi_{i}\text{ } \kappa_{i}(-1)\right]\E^{(s)}\left[F_s(q)\right]
\end{align*}
as the velocity is common and deterministic by Proposition \ref{prop common}.
\end{proof}
The next proposition claims the moment version of the Ghirlanda-Guerra identities under the stability hypothesis.
\begin{prop}
Let $\psi$, $F_s$ and $(\xi,Q)$ be as in Proposition \ref{prop AC}. If $(\xi,Q)$ is also ergodic under the considered evolutions, then for any $s\in\N$ its law satisfies
\begin{equation}
\E^{(s+1)}\left[q^r_{s,s+1}F_s(q)\right]=\frac{1}{s}\E^{(2)}[q_{12}^r]\E^{(s)}[F_s(q)]+\frac{1}{s}\sum_{l=1}^{s-1}\E^{(s)}[q^r_{ls}F_s(q)]
\label{eqn gg2}
\end{equation}
\end{prop}
\begin{proof}
As in the proof of Proposition \ref{prop AC}, $(\xi,Q)$ must be quasi-stationary under $\Phi_{r}$ for all linear functions $\lambda\kappa$ for $\lambda$ in an interval containing $0$.
In particular, it fulfills the hypothesis of Lemma \ref{factorization}. We take the $\lambda$-derivative on both sides of the identity \eqref{eqn factor}. A quick computation of the gaussian derivative of the left-hand side is possible as Proposition \ref{prop AC} and the factorization property show that only the terms where $\kappa_{i_1}$ is hit by the derivative are relevant. The straightforward calculation yields for the left-hand side
$$\sum_{l=1}^s\E^{(s)}\left[q_{ls}^{r}F_s(q)\right]-s\E^{(s+1)}\left[q_{s,s+1}^{r}F_s(q)\right].$$ 
The derivative of the r.h.s is simply $(1-\E^{(2)}[q_{12}^{r}])\E^{(s)}\left[F_s\right]$
by Lemma \ref{ppar}. The claim follows by combining both sides.
\end{proof}
Theorem \ref{thm main2} is now an easy corollary of the proposition. 
\begin{proof}[Proof of Theorem \ref{thm main2}]
By the hypothesis of the theorem, the ROSt is quasi-stationary under $\Phi_r$ for all $r\in\N$. In particular, the identities \eqref{eqn gg2} hold for every $r\in\N$ and hence for the distribution conditioned on the $\sigma$-field $\F_s$ generated by the overlaps of $s$ points
$$ 
\prob^{(s+1)}\left(q_{s,s+1}\in A \ | \F_s\right)=\frac{1}{s}\prob^{(2)}(q_{12}\in A)+\frac{1}{s}\sum_{l=1}^{s-1}\chi_A(q_{ls})
$$
where $A\subseteq [-1,1]$.
On the other hand, $\xi$ is independent of $Q$ by Theorem \ref{thm general}. Therefore, equation \eqref{eqn gg2} actually holds for every fixed integer $i_1,...,i_s$
$$
\prob\left(q_{i_s,i_{s+1}}\in A\ |\F_s\right)=\frac{1}{s}\prob(q_{i_1i_2}\in A)+\frac{1}{s}\sum_{l=1}^{s-1}\chi_A(q_{i_l,i_s}).
$$
Moreover, we know that given the directing measure $\mu$ on $\hilbert$, $Q$ is constructed as the Gram matrix of iid $\mu$-distributed elements. Hence the above can be rewritten as
$$
\E\left[ \bigotimes_{t=1}^s\mu\left(q_{s,s+1}\in A \right)\ \Big| \F_s\right] =\frac{1}{s}\E\left[\mu\otimes\mu\left(q_{12}\in A\right)\right]+\frac{1}{s}\sum_{l=1}^{s-1}\chi_A(q_{ls})
$$
and the theorem is proven.


\end{proof}
\appendix

\section{The evolution $\Phi$ revisited}
\label{app ext}
In this section, we prove Lemma \ref{lem exist} on the existence of a $\Lambda$-invariant probability measure on $\Omega=\Omega_{os}\times\prod_{t\in\Z}\E^{\N}$ which extends the law of a quasi-stationary ROSt. The exchangeability of the past time-steps of the evolution is also shown. We split the proof into two lemmas.
 \begin{lem}
Let $(\xi,Q)$ be a quasi-stationary ROSt under $\Phi_{r}$. There exists a unique $\Lambda$-invariant probability measure on $\Omega$ whose restriction on $\Omega_{os}$ is the law of $(\xi,Q)$. Moreover, this measure is ergodic under $\Lambda$ if and only if $(\xi,Q)$ is ergodic.
\label{lem1 exist}
\end{lem}
\begin{proof}
For convenience, we denote the evolution $\Phi_{\psi(\kappa(t))}$  by $\Phi_{t}$ to lighten notation. We also write $\Lambda$ for the map on the space $\Omega_{-T}:=\Omega_{os}\times\prod_{t\geq -T} \R^\N$ whose action is to evolve the configuration recording the present increment as the last one:
\begin{align*}
\Lambda:\Omega_{-T}&\to \Omega_{-T-1}\\
(\xi,Q,(\kappa(t),t\geq -T))&\mapsto \left(\Phi_{\psi(\kappa(0))}(\xi,Q),(\kappa_\downarrow(t+1),t\geq -T-1)\right)
\end{align*}
First, consider the collection of measures $\prob_r^{(T)}:=\prob_r\circ\Lambda^{-T}$, $T\in\N$, where $\prob_r$ is of the form \eqref{forw}.
We will prove that these measures are consistent: for all $T\in\N$,
\begin{equation}
\prob^{(T+1)}_r\Big|_{\Omega_{-T}}=\prob_r^{(T)}.
\label{eqn exist}
\end{equation}
The extension of $\prob_r$ to $\Omega$ then follows by Kolmogorov's extension theorem. By definition, $\prob_r^{(T)}$ is the distribution of 
\begin{equation}
\left(\Phi_{T-1}\circ...\circ \Phi_{0}(\xi,Q), (\kappa_\downarrow(t+T), t\geq -T )\right)
\label{eqn prop exist}
\end{equation}
under $\prob_r$. Similarly, $\prob^{(T+1)}_r$ restricted to $\Omega_{-T}$ corresponds to the distribution of
$$
\left(\Phi_{T}\circ...\circ\Phi_{1}\left(\Phi_{0}(\xi,Q)\right), (\kappa_\downarrow(t+T+1),t\geq -T) \right)\ .
$$
By stationarity, 
$ \Phi_{0}(\xi,Q)$, has the same distribution as $(\xi,Q)$ though its law depends explicitly on $\kappa(0)$. However, as the field $(\kappa(t+T+1),t\geq -T )$ depends only on $\kappa(0)$ through $Q$ and as the distribution of $Q$ is preserved under evolution, we have that the restriction of $\prob_r^{(T+1)}$ is the law of
$$ \left(\Phi_{T}\circ...\circ\Phi_{1}(\xi,Q), (\kappa_\downarrow(t+T+1),t\geq -T)\right)$$
which only differs from \eqref{eqn prop exist} by a mere relabeling of $t$.
Equation \eqref{eqn exist} is established and the existence is proven. The invariance under $\Lambda$ is straightforward from the construction of the measure. Moreover, the extension is ergodic as it is extremal in the set of $\Lambda$-invariant measure if and only if the law of $(\xi,Q)$ is extremal.
\end{proof}

\begin{lem}
The sequence of past increments $(\kappa(t),t<0)$ is exchangeable conditionally on $(\xi,Q)$ under the probability measure constructed in Lemma \ref{lem1 exist}.

Let $\alpha$ be the empirical measure of $(\kappa(-t),t\in\N)$. The random variables $\psi(\kappa_i(-t))$ have finite $p$-moments under the probability measure $\prob_r( \ \cdot \ |\xi,Q,\alpha)$ for any $i,t \in\N$ and $1\leq p<\infty$ a.s.
\label{lem2 exist}
\end{lem}
\begin{proof}
Denote by $\sigma(X)$ the $\sigma$-algebra generated by a random variable $X$. Define $S_i(T-1):=\sum_{t=0}^{T-1}\psi(\kappa_i(t))$ where the indexing $i$ is done through the ordering at time $0$. We claim that
\begin{equation}
\sigma\Big(\xi,Q,(S_i(T-1),i\in\N)\Big)=\sigma\left(\Phi_{T-1}\circ...\circ\Phi_{0}(\xi,Q),(\tilde{S}_j(T-1),j\in\N)\right)
\label{exch eqn}
\end{equation}
where $(\tilde{S}_j(T-1),j\in\N):=(S_i(T-1),i\in\N)_\downarrow$ are the increments of the $T$ time-steps reindexed with respect to the ordering after evolution. To shorten notation, let us write $\mathcal{G}$ for the left-hand side and $\tilde{\mathcal{G}}$ for the right-hand side. For convenience, we write $\tilde{\xi}$ for the evolved $\xi$ after $T$ time-steps i.e.
$$\tilde{\xi}:=\left(\frac{\xi_ie^{ S_i(T-1)}}{\sum_k\xi_ke^{ S_k(T-1)}},i\in\N\right)_\downarrow.$$
It is clear from the above expression that $\tilde{\xi}$ is $\mathcal{G}$-measurable. As the reindexing of $Q$ and $S_i(T-1)$ induced by the evolution depends only on $\tilde{\xi}$, we see that actually $\Phi_{T-1}\circ...\circ\Phi_{0}(\xi,Q)$ and $(\tilde{S}_j(T-1),j\in\N)$ are $\mathcal{G}$-measurable. The $\supseteq$ part of equation \eqref{exch eqn} is proven. For the $\subseteq$ part, it is easy to check that
$$\xi=\left(\frac{\tilde{\xi}_je^{-\tilde{S}_j(T-1)}}{\sum_k \tilde{\xi}_ke^{-\tilde{S}_k(T-1)}},j\in\N\right)_\downarrow.$$
Similarly as before, we conclude that $(\xi,Q)$ and $(S_i(T-1),i\in\N)$ are $\tilde{\mathcal{G}}$-measurable. Equation \eqref{exch eqn} is proven

Recall that the fields $\kappa(t)$ , $0\leq t\leq T-1$, indexed by the ordering at time $0$ are iid-distributed conditionally on $(\xi,Q)$. In particular, they are exchangeable given the sums $(S_i(T-1),i\in\N)$. Therefore, for any permutation $\rho$ of $T$ elements, the following holds
$$\prob_r(\kappa(t)\in A_t, \ 0\leq t\leq T-1 \ |\mathcal{G})=\prob_r(\kappa(\rho t)\in A_t, \ 0\leq t\leq T-1 \ |\mathcal{G})$$
for any $A_t$, $ 0\leq t\leq T-1 $, Borel sets of $\R^\N$. Moreover, the fields $\kappa(t)$ can be indexed with the ordering at time $T$ as this ordering is $\mathcal{G}$-measurable. From \eqref{exch eqn}, it follows that
$$\prob_r(\kappa(t)\in A_t,\ 0\leq t\leq T-1 \ |\tilde{\mathcal{G}})=\prob_r(\kappa(\rho t)\in A_t,\ 0\leq t\leq T-1 \ |\tilde{\mathcal{G}}).$$
The first claim is obtained from the above by integrating over $(\tilde{S}_j(T-1),j\in\N)_\downarrow$ and using invariance under $\Lambda$.

For the second claim, we can assume without loss of generality that $p$ is an integer.
By exchangeability in $t$, it suffices to prove the claim for $\psi(\kappa_i(-1))$, $i\in\N$.
The conclusion will be obtained by proving that $ \E_r\left[\ \sum_i\xi_i \ |\psi(\kappa_i(-1))|^p \ \right] <\infty$.
We have by definition of the past increment
$$
\E_r\left[\sum_i\xi_i \text{ }|\psi(\kappa_i(-1))|^p\right] 
= \E_r\left[\frac{\sum_i\xi_ie^{\psi(\kappa_i(0))}\text{ } |\psi(\kappa_i(0))|^p}{\sum_j\xi_j\text{ }e^{\psi(\kappa_j(0))}}\right].
$$
The Cauchy-Schwarz inequality followed by applications of Jensen's inequality with the functions $f(y)=y^2$ and $f(y)=1/y^2$ shows that the right-hand side is smaller than
$$
 \E_r\left[\sum_i\xi_i e^{2\psi(\kappa_i(0))}\text{ }\psi(\kappa_i(0))^{2p}\right]^{1/2}\E_r\left[\sum_i\xi_i e^{-2\psi(\kappa_i(0))}\right]^{1/2}.
$$
As $\kappa$ is independent of $\xi$ conditionally on $Q$, we can take the expectation over each $\kappa_i$ through to get $\left(\frac{d^{2p}}{d^{2p}}g(2)\right)^{1/2}g(-2)^{1/2}$ where $g(\lambda):=\int_\R\frac{e^{-z^2/2}}{\sqrt{2\pi}}e^{\lambda\psi(z)}dz$. But this is finite whenever $\psi$ satisfies Assumption \ref{intro ass}.
\end{proof}

\section{Reduction to the linear case}
\label{app var}
The proof of the main theorem in \cite{argaiz} was achieved by reducing the evolution with a smooth $\psi$ to an evolution with a linear $\psi$ by a central limit theorem argument. In brief, one considers $T$ independent steps of the evolution 
\begin{equation}
\Phi_{\lambda\psi(\kappa(T-1))}\circ ... \circ \Phi_{\lambda\psi(\kappa(0))}
\end{equation}
together with the scaling $\lambda\to \lambda/\sqrt{T}$. In the limit $T\to\infty$, the dynamics has simply gaussian increments with an effective covariance $\hat{q}_{ij}:=\E[\psi(\kappa_i)\psi(\kappa_j)]$. We could conclude that the $Q$-factors of $(\xi,\hat{Q})$ are RPC's from the analysis of the linear case. Monotonicity of the function $q_{ij}^r\mapsto \hat{q}_{ij}(r)$ for $r$ large enough and properties of the RPC's permitted to deduce that $(\xi,Q)$ is a RPC whenever $(\xi,\hat{Q})$ is. 
A similar reduction to the linear case can be carried when quasi-stationarity is assumed for a collection of functions $\psi(\beta\cdot +h)$. Under the new assumption, the limiting linear dynamics turns out to be somewhat simpler as it produces the same effective covariance matrix as the original system. The proof is very similar to the proof of Lemma 4.8 in \cite{argaiz}. We present it for completeness.
\begin{lem}
Let $h\in\R$ be such that $\psi'(h)\neq 0$. If $(\xi,Q)$ is a quasi-stationary ROSt under $\Phi_r$ with function $\psi(\beta\kappa+h)$ for all $\beta$ in an interval containing $0$, then $(\xi,Q)$ is also quasi-stationary under $\Phi_r$ with function $\lambda\kappa$ for all $\lambda$ in an interval containing $0$.
\label{lem reduction}
\end{lem}
\begin{proof}
First, we recall that the law of a ROSt is determined by the class of continuous functions that depend only on a finite number of points (Proposition 1.2 in \cite{argaiz}). Let $f:\Omega_{os}\to \R$ be a continuous function depending on the first $n$ points for some $n\in\N$ i.e. $f(\xi,Q)=f(\xi_1,...,\xi_n; Q_n)$ where $Q_n=\{q_{ij}\}_{1\leq i,j\leq n}$. Consider $T$ independent copies of the gaussian field $\kappa$: $(\kappa(t),0\leq t\leq T-1)$. Define the evolution by $T$ independent steps
\begin{equation}
\Phi_T:=\Phi_{\psi(\beta\kappa(T-1)+h)}\circ ... \circ \Phi_{\psi(\beta\kappa(0)+h)} \ .
\label{eqn1 app}
\end{equation}
 To prove the claim, we need to show that for any such $f:\Omega_{os}\to\Omega_{os}$ and under an appropriate scaling of $\beta$
\begin{equation}
\E_r[f(\xi,Q)]=\lim_{T\to\infty}\E_r[f(\Phi_T(\xi,Q))]=\E_{r}[f(\Phi_{\lambda\kappa}(\xi,Q))]
\label{eqn reduction}
\end{equation}
for some $\lambda\in\R$. The first equality holds by the quasi-stationarity hypothesis for all $\beta$ in a neighborhood of $0$. We prove the second one.

We choose the scaling
$$ \beta=\beta(T)=\frac{\lambda}{|\psi'(h)| \sqrt{T}}\ .$$
It is straightforward to check, by expanding $\psi$ around $h$ and using the boundedness of the second derivatives, that with this choice
$$ 
\lim_{T\to\infty}\sum_{t=0}^{T-1}\E_r\Big[\big(\psi(\beta\kappa_i(t)+h)-\psi(h)\big)\big(\psi(\beta\kappa_j(t)+h)-\psi(h)\big)\ \Big| \ Q\Big]
= \lambda^2q_{ij}^r \ .
$$
Note that, because of the normalization of the dynamics, the effective increment of each particle can be taken to be $\psi(\beta\kappa_i+h)-\psi(h)$. Hence, by the finite-dimensional central limit theorem and the above convergence, the increments of a fixed number of particles converge to a centered gaussian field with covariance matrix $\lambda^2 q_{ij}^r$.
It remains to prove that the limit $T\to\infty$ of \eqref{eqn reduction} is well-approximated by considering a large but finite number of particles. 

For $\delta',\delta\in(0,1]$ and $\delta'<\delta$, we define the function $f_{\delta}$ and $f_{\delta,\delta'}$ as
$$ f_\delta(\xi_1,...,\xi_n;Q_n):=f(\xi_1,...,\xi_n;Q_n)\chi_{\{\xi_n\geq \delta\}}$$
and
$$ f_{\delta,\delta'}(\xi_1,...,\xi_n;Q_n):=f_\delta(\xi_1/N_{\delta'},...,\xi_n/N_{\delta'};Q_n) $$
where $ N_{\delta'}:=\sum_{i:\xi_i\geq \delta'}\xi_i.$
Clearly, $f_\delta\to f$ a.s. when $\delta\to 0$ as $\xi_n>0$ a.s. Notice also that $N_{\delta'}\to 1$ when $\delta'\to 0$. Therefore, by continuity
$$ \lim_{\delta\to 0}\lim_{\delta'\to 0}f_{\delta,\delta'}(\xi_1,...,\xi_n;Q_n)=f(\xi_1,...,\xi_n;Q_n) \text{ a.s.}$$
Let $A_{N,\delta',T}^c$ be the event that all evolved points in $[\delta',1]$ after $T$ steps come from the first $N$ before evolution. We write $\Phi_r(\xi,Q)|_N$ for the evolution restricted to the first $N$ points of $(\xi,Q)$. 
Because the function $f_{\delta,\delta'}(\Phi_r(\xi,Q))$ on the event $A^c_{N,\delta',T}$ is effectively a function of $\Phi_r(\xi,Q)|_N$, one has
\begin{multline}
\Big|\E_r[f_{\delta,\delta'}(\Phi_T(\xi,Q))]-\E_r\left[f_{\delta,\delta'}(\Phi_T(\xi,Q)|_N)\right]\Big|\leq
\\ \E_r\left[\big|f_{\delta,\delta'}(\Phi_T(\xi,Q))-f_{\delta,\delta'}(\Phi_T(\xi,Q)|_N)\big|\chi_{A_{N,\delta',T}}\right].
\label{eqn1 evolution estimate}
\end{multline}
The limit \eqref{eqn reduction} will thus hold by respectively taking the limits $T\to\infty$, $N\to\infty$ and $\delta,\delta'\to 0$ if we can show that the probability of the event $A_{N,\delta',T}$ is small for $N$ large uniformly in $T$. But this is clear from the fact that under the chosen scaling of $\beta$  (see Lemma 4.6 of \cite{argaiz} for details) : 
$$ \prob_r\left(A_{N,\delta',T}\right)\leq K\sum_{i>N}\E[\xi_i]$$
for some constant $K$ that only depends on $\psi$, $\delta'$ and $\lambda$. 
\end{proof}


\end{document}